\documentclass[12pt,a4paper]{article}
\usepackage[top=2.5cm,bottom=2.5cm,left=2.5cm,right=2.5cm]{geometry}

\usepackage{amssymb,amsmath,graphicx,color,amsthm,centernot,enumerate}
\usepackage[capitalise]{cleveref}
\usepackage[labelformat=simple]{subcaption}
\usepackage[font=small,labelfont=bf,width=12.5cm]{caption}

\usepackage{multirow,tabularx}

\newtheorem{theorem}{Theorem}[section]
\newtheorem{lemma}[theorem]{Lemma}
\newtheorem{corollary}[theorem]{Corollary}

\newtheorem{observation}{Observation}
\newtheorem{conjecture}[theorem]{Conjecture}

\newcommand{\set}[1]{\ensuremath{\left\{#1 \right\}}}

\newcommand{\out}[1]{\ensuremath{\mathrm{out}(#1)}}

    {\noindent \emph{Proof.} {}{#1}{}}{$~$\hfill $~\blacklozenge$ \vspace{0.2cm}}

\definecolor{defblue}{rgb}{0.4,0,0.84}
\definecolor{greyblue}{rgb}{0.23,0.4,0.70}
\definecolor{orange}{rgb}{1.0,0.5,0.2}
\definecolor{violet}{rgb}{0.55,0,0.55}

\usepackage{url}
\makeatletter
\g@addto@macro{\UrlBreaks}{\UrlOrds}
\makeatother

\newcolumntype{Y}{>{\centering\arraybackslash}X}

\begin{document}

\title{{\bf Degree-balanced decompositions \\ of cubic graphs}}

\author
{
	Borut Lu\v{z}ar\thanks{Faculty of Information Studies in Novo mesto, Slovenia.} \thanks{Rudolfovo Institute, Novo mesto, Slovenia.} \quad
	Jakub Przyby{\l}o\thanks{AGH University of Krakow, 
 al. A. Mickiewicza 30, 30-059 Krakow, Poland.} \quad
	Roman Sot\'{a}k\thanks{Pavol Jozef \v Saf\'{a}rik University, Faculty of Science, Ko\v{s}ice, Slovakia.} 	
}

\maketitle

{
\begin{abstract}
	We show that every cubic graph on $n$ vertices contains a spanning subgraph
	in which the number of vertices of each degree
 deviates from $\frac{n}{4}$ by at most $\frac{1}{2}$,
	up to three exceptions.
	This resolves the conjecture of Alon and Wei 
	({\em Irregular subgraphs, Combin. Probab. Comput. 32(2) (2023), 269--283})
	for cubic graphs.
\end{abstract}
}

\medskip
{\noindent\small \textbf{Keywords:} irregular subgraph, repeated degrees, 
degree-balanced decomposition.}

\section{Introduction}

In this paper, we only consider simple graphs (i.e., graphs with no loops and no parallel edges).
Given a graph $G$ and a nonnegative integer $k$,
we denote by $m(G,k)$ the number of vertices of degree $k$ in $G$,
and by $m(G)$ the maximum number of vertices of the same degree in $G$.

Recently, Alon and Wei~\cite{AloWei23} considered the problem of searching for a spanning subgraph $H$ of a $d$-regular graph $G$, 
for which $m(H)$ is the smallest possible. 
Clearly, the lower bound for this value is $m(H) \ge \lceil\frac{n}{d+1}\rceil$.
The authors of~\cite{AloWei23} suspect that the best general upper bound for such $m(H)$ is however very close to the same quantity.
More specifically, 
they proposed the following two conjectures.
\begin{conjecture}[Alon \& Wei~\cite{AloWei23}]
	\label{conj:regular}
	Every $d$-regular graph $G$ on $n$ vertices contains a spanning subgraph $H$ 
	such that for every $k$, $0 \le k \le d$, 
	$$
		\left | m(H,k) - \frac{n}{d+1} \right | \le 2\,.
	$$
\end{conjecture}

\begin{conjecture}[Alon \& Wei~\cite{AloWei23}]
	\label{conj:general}
	Every graph $G$ on $n$ vertices and minimum degree $\delta$ contains a spanning subgraph $H$ 
	satisfying 
	$$
		m(H) \le \frac{n}{\delta+1} + 2\,.
	$$
\end{conjecture}

These conjectures thus address the quest for a highly irregular subgraph $H$ in a given graph $G$, that is a subgraph with degrees nearly as diversified as possible, 
almost as much as degrees in the host graph $G$ permit.
Alon and Wei~\cite{AloWei23} managed to prove
the conjectures up to constant multiplicative factors.
Namely, 
they assured the existence of spanning subgraphs $H$ with $m(H) \le 8\frac{n}{d} + 2$ and  $m(H) \le 16\frac{n}{\delta} + 4$ 
in all $n$-vertex $d$-regular graphs and graphs with minimum degree $\delta>0$, respectively.
They also confirmed Conjecture~\ref{conj:regular} asymptotically for small enough $d$ compared to $n$, showing that each $n$-vertex $d$-regular graph with 
$d  = o((n/\log n)^{1/3})$
contains a spanning subgraph $H$ with $m(H,k) = (1 + o(1))\frac{n}{d+1}$,
for every $0 \le k \le d$.
Later, Fox, Luo, and Pham~\cite{FoxLuoPha24}  improved this result, by significantly extending the range of $d$ for which its statement holds, 
i.e., towards all $d = o(n/(\log n)^{12})$.
Finally, Ma and Xie~\cite{MaXie24} provided the first bound for Conjecture~\ref{conj:regular} independent of $n$.
\begin{theorem}[Ma \& Xie~\cite{MaXie24}]
	Every $d$-regular multigraph $G$ on $n$ vertices contains a spanning subgraph $H$ 
	such that for every $k$, $0 \le k \le d$, 
	$$
		\left | m(H,k) - \frac{n}{d+1} \right | \le 2d^2\,.
	$$	
\end{theorem}
For general $n$-vertex graphs with $\delta>0$, Alon and Wei~\cite{AloWei23} provided spanning subgraphs $H$ with $m(H) \le (1+o(1)) \lceil \frac{n}{\delta+1} \rceil + 2$ 
and confirmed Conjecture~\ref{conj:general} up to additive factor $1$ for sufficiently large $n$ with $\delta^{1.24}\geq n$.

Apart from Ma and Xie, who used a novel deterministic approach to prove their result,
the above-mentioned asymptotic approximations of Conjectures~\ref{conj:regular} and~\ref{conj:general}
were based on an analysis of a randomized procedure of choosing $H$, 
exploited earlier, e.g., in~\cite{Frieze,AsymptoticIrregStrReg,PrzybyloWeiLong,PrzybyloWeiShort} 
to investigate a related concept, the so-called irregularity strength of graphs.
Within this problem, instead of looking for
a subgraph with limited frequency of every potential degree, one strives to assure at most one vertex of each degree.
This is achieved either by multiplying the edges of the host graph $G$ or equivalently weighting its edges with positive integers and considering weighted degrees. 
The least $k$ admitting such a weighting with the maximum weight $k$ is called the \emph{irregularity strength} of $G$ and denoted $s(G)$. 
This notion was introduced 
in~\cite{ChaJacLehOelRuiSab88} in the 80s,
along with the key open problem 
of the related field
from~\cite{FauLeh87}, 
which asserts that $s(G)\le \frac{n}{d}+C$ for every $d$-regular graph $G$ of order $n$,
that is just a constant above a trivial lower bound.
See, e.g.,~\cite{PrzybyloWeiLong} for an extensive list of references to papers devoted to studying this conjecture. 
These, as well as the main objective of our paper, are related with problems concerning the so-called degree-constrained subgraphs, 
see, e.g.,~\cite{AddDalRee08} (and also~\cite{AddAldDalRee05,AddDalMcDReeTho07}), 
whose results were in particular exploited by Alon and Wei in~\cite{AloWei23}. 
In the same paper, they also established several more direct relations 
between the irregularity strength of graphs and the existence 
of almost irregular subgraphs we focus on.

In this paper, instead of asymptotics, we investigate exact bounds 
for $3$-regular (i.e., {\em cubic}) graphs.
In particular, we prove that Conjecture~\ref{conj:regular} itself holds for such graphs, taking the first modest step towards solving 
the conjecture in its literal form. 
It was Alon and Wei~\cite{AloWei23} who observed that one cannot replace the constant $2$ in Conjecture~\ref{conj:regular} with $1$, as exemplifies, e.g.,
the graph comprised of 
two components isomorphic to the cycle of length $4$.
They however suspected that such a strengthening might be possible for every $d$ up to a limited number of small exceptions.
Such strengthening of Conjecture~\ref{conj:regular}  
for $d=3$ is exactly implied by our main result below, which is even slightly stronger, and in face of the mentioned trivial lower bound is optimal for all cubic graphs.
\begin{theorem}
	\label{thm:main}
	Every cubic graph $G$ on $n$ vertices, not isomorphic to $K_4$, $K_{3,3}$, 
	or $3K_4$, contains a spanning subgraph $H$ 
	such that for every $k$, $0 \le k \le 3$, 
	$$
        m(H,k) \in \left\{\left\lfloor\frac{n}{d+1}\right\rfloor,
        \left\lceil\frac{n}{d+1}\right\rceil\right\}.
	$$
\end{theorem}

We note that Conjecture~\ref{conj:regular} for the case $d=3$ was independently confirmed
by Ma and Xie~\cite{MaXie24}, who employed a refinement of their local adjustments method.
Our result differs from theirs in two aspects. 
We exploit an entirely different technique, which is strictly constructive and yields a 
straightforward algorithm directly generating the desired subgraph.
But more importantly, our result for cubic graphs is stronger than the one in~\cite{MaXie24}, 
which basically confirms Conjecture~\ref{conj:regular} for this graph class, 
i.e., implies an upper bound of $2$ for the achievable maximum deviation of a subgraph degree frequency from $n/4$. 
In turn, we provide a complete solution for this invariant, 
determining the exact value of achievable maximum deviation for every cubic graph,
which, apart from a few exceptional cases, equals $0$ or $1/2$.

\section{Proof of Theorem~\ref{thm:main}}

In this section, we prove the main result of the paper.
We first present some additional terminology and auxiliary observations. 

We say that a $d$-regular graph $G$ is 
{\em $(n_d,n_{d-1},\dots,n_0)$-decomposable}
if there exists a spanning subgraph $H$ of $G$ 
such that $n_i = m(H,i)$, for every $i \in \set{0,\dots,d}$.

Since in the complement $\overline{H}$ of $H$ the degree of a vertex of degree $k$ from $H$ equals $d-k$, 
the following observation immediately follows.
\begin{observation}
	\label{obs:complement}
	If a $d$-regular graph is $(n_d,n_{d-1},\dots,n_0)$-decomposable,
	then it is also $(n_0,n_{1},\dots,n_d)$-decomposable.
\end{observation}

Given two subsets of vertices $X$ and $Y$ in a graph $G=(V,E)$, 
by $e(X,Y)$ and $e(X)$, we denote the number of edges having one end vertex in $X$ and the other in $Y$,
and having both endvertices in $X$, respectively.
By ${\rm out}(X)$ we denote the number of edges joining $X$ with $V\smallsetminus X$ in $G$. 
Moreover, given an edge-coloring of $G$,
by $e_c(X,Y)$ and $e_c(X)$ we denote the number of edges of color $c$ having one end vertex in $X$ and the other in $Y$,
and having both endvertices in $X$, respectively.
For a vertex $v$, we additionally denote by $N_c(v)$ the set of neighbors $u$ of $v$ in $G$ such that $uv$ is colored $c$.

\subsection{Connected graphs}

In order to handle later graphs with multiple components we first prove 
a strengthening of Theorem~\ref{thm:main} in the connected case,
implying the existence of specific types of decompositions.
\begin{lemma}
	\label{thm:connected}
	Let $G$ be a connected cubic graph on $n$ vertices.
	The following statements hold:
	\begin{enumerate}[(i)]
	\item{} if $n = 4t$ and $G$ is not isomorphic to $K_4$, then $G$ is $(t,t,t,t)$-decomposable;
	\item{} if $n = 4t$, then $G$ is $(t-1,t-1,t+1,t+1)$-decomposable;
	\item{} if $n = 4t+2$ and $G$ is not isomorphic to $K_{3,3}$, then $G$ is $(t,t+1,t,t+1)$-decomposable;
	\item{} if $n = 4t+2$, then $G$ is $(t-1,t,t+1,t+2)$-decomposable.
	\end{enumerate}
\end{lemma}

\begin{proof}
	Let $G$ be a connected cubic graph and let $n = |V(G)|$.
	We will show that for each of the statements $(i)$-$(iv)$, we can construct a spanning subgraph $H$ of $G$ 
	with a suitable number of vertices of each degree.
	In particular, we will color the edges of $G$ with colors $0$ and $1$, 
	and the edges of color $1$ will represent the edges of $H$.
	
	We consider the cases with $t=1$ separately.
	Suppose first that $t = 1$ and $n=4t$. Then $G$ is isomorphic to $K_4$ and 
	the graph $H = 2K_1 \cup K_2$ realizes the statement $(ii)$.
	Suppose now that $t = 1$ and $n=4t+2$.
	There are two cubic graphs on $6$ vertices: $K_{3,3}$ and the $3$-prism (i.e., the complement of $C_6$).
	If $G$ is isomorphic to $K_{3,3}$, then the graph $H = 3K_1 \cup P_3$ 
	realizes the statement $(iv)$.
	Otherwise $G$ is isomorphic to the $3$-prism.
	The statement $(iii)$ is realized by the graph $H$ being isomorphic 
	to a triangle with one pendant edge and two isolated vertices,
	and the statement $(iv)$ is realized by the graph $H = 3K_1 \cup P_3$.
	
	So, we may assume that $t \ge 2$ and thus $n \ge 8$.
	We first color all the edges of $G$ with $0$, 
	and in the following steps, we carefully recolor some of them with color $1$.
	We call a vertex with $k$ incident edges of color $1$ a {\em $k$-vertex},
	and we denote the set of all $k$-vertices by $V_k$; note that after recoloring edges 
	with $0$ or $1$, we always update the sets $V_i$ accordingly.	
	Note also that, by the Handshaking Lemma, the sizes of $V_3$ and $V_1$ always have the same parity.
	By $n_i$ we denote the target number of $i$-vertices in $H$.

 Our argument is divided into several natural stages.

	\bigskip\noindent
	{\bf Stage 1: Determining $3$-vertices.} \quad
		We choose the vertices for $V_3$ one by one; 
		namely, we start by finding a shortest cycle $C = v_1v_2\dots v_g$ in $G$,
		where $g$ is the length of $C$.
				
		We first claim that $g \le n/2$.		
  		Suppose to the contrary that $g > n/2$.
		Since $n \ge 8$, then $g \ge 5$.
		By the minimality of $C$, it has no chords and no two vertices of $C$ can have a common neighbor outside $C$. Hence, 
		each of the $g$ vertices on $C$ has its unique neighbor not in $C$,
		and thus $2g \le n$, a contradiction.
		
		We continue by coloring all edges incident to $v_1$ with $1$,
		and then proceed with recoloring the edges incident with the vertices $v_2$, $v_3$, and so on along $C$
		until $|V_3| = n_3$ or $V_3 = V(C)$.
		Note that since $G$ is not isomorphic to $K_4$ and $C$ is a shortest cycle,
		at every step exactly one vertex is introduced to $V_3$.
		
		Now, if $g < n_3$, then we continue recoloring to $1$ all edges incident 
		with any vertex adjacent to a vertex from $V_3$
		which is not adjacent to any vertex from $V_2$, until $|V_3| = n_3$.
		Note that we can always choose such a vertex $v$. 
  		Indeed, this is obvious if $e_0(V_2,V_1\cup V_0)>0$. 
		Otherwise, as 
		$|V_3|+|V_2|<n_3+n_2<n$ and $G$ is connected, there is a neighbor $v\in V_1$ of some vertex in $V_3$; since $e_0(V_2,V_1\cup V_0)=0$, $v$ has no neighbors in $V_2$.
		Therefore, in each step, the size of $V_3$ increases by $1$ 
		and we can achieve the target size of $V_3$.

		The choice of the vertices for $V_3$ guarantees that 
		the graph induced by $V_3$ is connected. 
		This means that $e(V_3) \ge n_3 - 1$, 
		and if $g \le n_3$, then there is a cycle in $V_3$,
		giving $e(V_3) \ge n_3$.

		We now compute the number of edges $\out{V_3}$ with exactly one end vertex in $V_3$.		
		By counting the half-edges incident with the vertices in $V_3$, we infer
		$$
			\out{V_3} = 3n_3 - 2 e(V_3) \le 3n_3 - 2 (n_3-1) = n_3 + 2\,.
		$$	
		Therefore, after completing the set $V_3$, we have that 
		\begin{equation*}
			\label{eq:V1}
			2|V_2| + |V_1| = \out{V_3} \le n_3 + 2\,.
		\end{equation*}		
		It follows that 
		\begin{equation}
			\label{eq:V2}
			|V_2| \le n_2 \qquad {\rm and} \qquad |V_1| \le n_3 + 2 \le n_1 + 2\,.
		\end{equation}
		Since in the cases $(ii)$ and $(iv)$, we have $n_1 = n_3 + 2$, 
		this implies that in cases $(ii)$ and $(iv)$, $|V_1| \le n_1$.
		Moreover, if $g \le n_3 + 1$, then all the edges of $C$ are colored with $1$, 
		and thus $|V_1| \le n_3 \le n_1$.
		
	\bigskip\noindent
	{\bf Stage 2: Determining $2$-vertices.} \quad
		We continue by completing the set $V_2$, which may already contain some vertices from the previous stage.				
		First, we define three rules for choosing $2$-vertices.
		\begin{itemize}
			\item[$(R_1)$] If $|V_2| < n_2 - 1$ and there is a vertex $v \in V_1 \cap N_1(V_2 \cup V_3)$ adjacent to a vertex $u \in V_1$ 
				(hence, $uv$ has color $0$),
				then we color $uv$ by $1$; 
				thus increasing the size of $V_2$ by $2$ and decreasing the size of $V_1$ by $2$.
			\item[$(R_2)$] If $|V_2| < n_2$ and there is a vertex $v \in V_1$ with a neighbor $u \in V_0$,
				then we color $uv$ by $1$; 
				thus increasing the size of $V_2$ by $1$, decreasing the size of $V_0$ by $1$, and retaining the size of $V_1$.
			\item[$(R_3)$] If $|V_2| < n_2$, $|V_1| < n_1$, and there is a vertex $v \in V_0$ with two neighbors $u,w \in V_0$,
				then we color $vu$ and $vw$ with $1$;
				thus increasing the size of $V_2$ by $1$, increasing the size of $V_1$ by $2$, and decreasing the size of $V_0$ by $3$.
		\end{itemize}
		
		Next, until $|V_2| = n_2$, we repeatedly apply the rule $R_1$ if possible, 
		and otherwise the rule $R_2$ if possible.
		Moreover, when $R_2$ is applied, we prefer coloring an edge of $C$ 
		if possible.
		
		We claim that if $|V_2| < n_2$ and we cannot apply rules $R_1$ or $R_2$, 
		then by~\eqref{eq:V2}, we have that $|V_1| \le n_1$.
		Recall that in the cases $(ii)$ and $(iv)$, as well as when $g \le n_3 + 1$,
		this holds even at the beginning of Stage~2,
		while applying $R_1$ or $R_2$ does not increase $|V_1|$.
		In the remaining two cases, $(i)$ and $(iii)$, 
		we have $n_3 + n_2 = n/2 \ge g$.
		Moreover, we also have that $g > n_3 + 1$.
		In this case, we have $|V_2|=0$ at the beginning of Stage~2.
		Suppose that we cannot apply $R_1$ until we apply $R_3$ or $|V_2| = n_2$.
		Then, by the rule, we apply $R_2$ and color the edge in $C$ first.
		This process will be repeated at most $g-2$ times.
		After that, $|V_2| = g - |V_3| - 2 \le n_2 - 2$,
		and there is only one edge in $C$ of color $0$ incident to two vertices in $V_1$.
		Hence, we apply $R_1$ on this edge.
		
		Suppose now that we cannot apply rules $R_1$ or $R_2$, 
		and we still have $|V_2| < n_2$. 
		We will show that we can apply $R_3$, meaning also that we retain $|V_1| \le n_1$ 
		(recall that the size of $V_1$ increases by $2$ at every application of $R_3$).
		
		Since $R_1$ and $R_2$ cannot be applied, 
		we have that there is no edge between $V_1$ and $V_0$ (the rule $R_2$ does not apply),
		and that either $e_0(V_1)=e(V_1) = 0$ or $|V_2| = n_2 - 1$ (the rule $R_1$ does not apply).
		We will show by contradiction that in any case, we can apply the rule $R_3$.
		So, suppose that $R_3$ cannot be applied. We consider two cases.

		\begin{figure}[htp!]
			$$
				\includegraphics{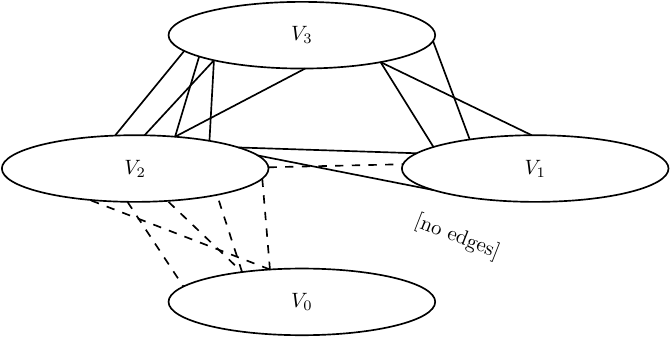}
			$$
			\caption{The graph $G$ in the case when $R_1,R_2$ do not apply, i.e. $e(V_1,V_0) = 0$.
				The edges of color $1$ and $0$ are depicted solid and dashed, respectively.}
			\label{fig:2}
		\end{figure}
  
		Suppose first that $e(V_1,V_0) = 0$ and $e(V_1) = 0$ (see Figure~\ref{fig:2}).
		Then, since $G$ is connected and thus 
		there is at least one edge between $V_2$ and $V_0$,
		and for each vertex $v \in V_1$, 
		we have $N_0(v) \subseteq V_2$,
		it follows that $2|V_1| < |V_2| < n_2$.
		Now, if there is a vertex $v \in V_0$ with two neighbors 
		$u,w \in V_0$, then we must have $|V_1| = n_1$.
		But in this case, $2n_1 < n_2$, which means that $t < 1$, a contradiction.		
		So, every vertex in $V_0$ has at most one neighbor from $V_0$ and therefore, $e(V_0,V_2) \ge 2|V_0|$.
		By the construction, we also have that $|V_0| > n_0$ and $|V_2| - 2|V_1| \ge e(V_0,V_2)$, which altogether gives
		$$
			n_2 > |V_2| \ge 2|V_1| + 2|V_0| \ge 2(n_0 + 1)\,,
		$$
		a contradiction.
				
		Second, suppose that $e(V_1,V_0) = 0$ and $|V_2| = n_2 - 1$ (see Figure~\ref{fig:2}).
		
		We first argue that $e(V_1) \le 2$.
		Indeed, if $e(V_1) \ge 1$ when $|V_2| \le n_2 - 2$, we would apply $R_1$.
		Otherwise, when $|V_2| \le n_2 - 2$, 
		we have that $e(V_1) = 0$ and
		either the rule $R_2$ is applied, meaning that by coloring the edge $vu$, 
		at most $2$ edges between vertices in $V_1$ are introduced 
		(in the case when $u$ had three neighbors from $V_1$ before coloring $uv$),
		or the rule $R_3$ is applied and at most $1$ edge between vertices in $V_1$ 
		is introduced (in the case when $v$ and $w$ are adjacent).
		
		Now, if every vertex in $V_0$ has at most one neighbor from $V_0$,
		then we again have $e(V_0,V_2) \ge 2|V_0| \ge 2(n_0+1)$, and moreover, since $|V_2| \ge e(V_0,V_2)$,
		we infer that 
		$$
			n_2 > |V_2| \ge 2(n_0+1)\,,
		$$
		a contradiction.
		
		Therefore, there is at least one vertex $v \in V_0$ with two neighbors $u,w \in V_0$
		and $|V_1| = n_1$ (recall that by the parity, if $|V_1| < n_1$, then also $|V_1| \le n_1 - 2$).
		By the connectivity of $G$, we have $|V_2| > e_0(V_2,V_1)$,
		and since $e(V_1) \le 2$, we have $e_0(V_2,V_1) \ge 2(n_1 - 2)$.
		Thus,
		$$
			n_2 - 1 = |V_2| > e_0(V_2,V_1) \ge 2(n_1 - 2)\,,
		$$
		and hence
  		\begin{equation}
			\label{eq:n2n1}
			n_2 > 2n_1 - 3\,.
		\end{equation}
		In the cases of small $n$, we do not reach a contradiction, 
		and thus we need to show separately that every cubic graph 
		with the given properties
		admits a suitable decomposition.
		We analyze~(\ref{eq:n2n1}) 
  for each of the four theorem's statements separately (recall that $t \ge 2$):
		\begin{enumerate}[(i)]
			\item{} $t > 2t-3$ \quad $\Rightarrow$ \quad $t < 3$. 
				However, this means that $t=2$, and since $|V_1| \le n_1$ once we potentially need to use R3, 
				there will be no 
    vertex $v \in V_0$ with three neighbors in $V_1$.
				Thus, $e(V_1) \le 1$ and $t > 2t - 1$, implying that $t < 1$.
			\item{} $t - 1 > 2(t+1)-3$ \quad $\Rightarrow$ \quad $t < 0$, 
				hence, this case is irrelevant.
			\item{} $t + 1 > 2t-3$ \quad $\Rightarrow$ \quad $t < 4$.		
				In the case with $t=2$, we argue similarly as in $(i)$, 
				and infer that $t < 2$. 
				So, we only need to consider the case with $t = 3$.
				This means that we need to find a $(3,4,3,4)$-decomposition of $G$.
				
				First, we claim that $e(V_1) = 2$. 
				Indeed, if $e(V_1) \le 1$, 
				then from $n_2 - 1 > 2(n_1-1)$,
				it follows that $t < 2$, a contradiction.
				
				Next, we already have that $|V_3|=|V_2|=|V_1|=3$,
				and so $|V_0| = 5$.
				Moreover, since $e(V_1,V_0)=0$ and $e(V_1) = 2$,
				we have that $e_0(V_2,V_1) = 2$, 
				and consequently $e(V_2,V_0) = 1$.
				Therefore, the five vertices of $V_0$
				must induce a subgraph $H'$ of $G$ isomorphic to $K_4$ 
				with one subdivided edge,
				where the unique vertex of degree $2$, call it $v$,
				is adjacent to a vertex $u$ from $V_2$.
				Let $u_1$ and $u_2$ be the two neighbors of $u$,
				distinct from $v$.
				Since $G$ is cubic with $14$ vertices, there is also a vertex $w$
				not adjacent with any vertex from $V_0$ or 
				with the vertices $u$, $u_1$. 
				Now, we recolor the edges of $G$
				using color $1$ for  
    all the edges of $H'$
				except one edge incident with $v$,				
				the edges $uv$, $uu_1$, 
				and two edges incident with $w$.
				We color the remaining edges with $0$ (see Figure~\ref{fig:special}).
				This gives a spanning subgraph of $G$ 
				with three $3$-vertices, four $2$-vertices,
				three $1$-vertices, and four $0$-vertices as required,
				and thus $G$ has a $(3,4,3,4)$-decomposition.
				\begin{figure}[htp!]
					$$
						\includegraphics{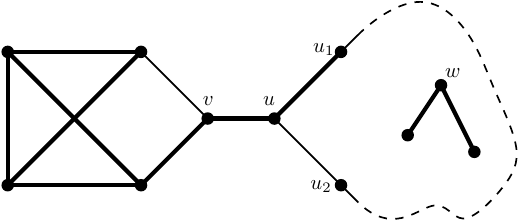}
					$$
					\caption{The graph $G$ in the case $(iii)$ when $t = 3$ (where neighbors of $w$ may coincide with
     $u_2$ or neighbors of $u_1,u_2$ other than $u$). 
						The bold edges depict the edges of color $1$.}
					\label{fig:special}
				\end{figure}
				
			\item{} $t > 2(t+1)-3$ \quad $\Rightarrow$ \quad $t < 1$, 
				hence, this case is irrelevant.
		\end{enumerate}
				
		This means that we can always attain the target number of vertices in $V_2$ 
		by using the three rules and that in all the cases $|V_1| \le n_1$.
	
	\bigskip\noindent
	{\bf Stage 3: Determining $1$-vertices.} \quad		
		This stage is realized 
  only if $|V_1| < n_1$,
		and so, by the parity, $|V_1| \le n_1 - 2$.
		Again, we add $1$-vertices iteratively without changing $3$- and $2$-vertices.
		This means that, in every step, we increase $|V_1|$ by $2$ 
		by coloring an edge between two $0$-vertices with $1$.
		
		Clearly, if there is an edge between two $0$-vertices, then we can increase $|V_1|$.
		Thus, we only need to consider the situation when $e(V_0) = 0$.
		In such a case, since $|V_3| = n_3$, $|V_2| = n_2$, and $|V_1| \le n_1 - 2$,
		we have that $|V_0| \ge n_0 + 2$, and so
		$$
			e(V_0, V_1 \cup V_2) = 3|V_0| \ge 3(n_0 + 2)\,.
		$$
		On the other hand,
		$$
			e(V_0, V_1 \cup V_2) \le n_2 + 2|V_1| \le n_2 + 2(n_1-2)\,.
		$$
  It thus follows that
    	$$
			3n_0 + 6 \le n_2 + 2n_1-4\,,
		$$
  which leads to a contradiction, as in all cases, $n_0\ge n/4$, $n_1,n_2\le n/4+1$.
		
		\bigskip		
		Therefore, we can always construct the sets $V_i$, for $i \in \set{0,1,2,3}$ 
		with the given sizes; this completes the proof.
\end{proof}

\subsection{Non-connected graphs}

Next, we prove that the decomposition properties guaranteed by Lemma~\ref{thm:connected}
hold for all cubic graphs, also those with more components,
with two additional exceptions.

\begin{theorem}
	\label{thm:general}
	Let $G$ be a cubic graph on $n$ vertices.
	The following statements hold:
	\begin{enumerate}[(i)]
	\item{} if $n = 4t$ and $G$ is not isomorphic to $K_4$ and $3K_4$, then $G$ is $(t,t,t,t)$-decomposable;
	\item{} if $n = 4t$ and $G$ is not isomorphic to $2K_4$, then $G$ is $(t-1,t-1,t+1,t+1)$-decomposable;
	\item{} if $n = 4t+2$ and $G$ is not isomorphic to $K_{3,3}$, then $G$ is $(t,t+1,t,t+1)$-decomposable;
	\item{} if $n = 4t+2$, then $G$ is $(t-1,t,t+1,t+2)$-decomposable.
	\end{enumerate}
\end{theorem}

\begin{proof}
	By Lemma~\ref{thm:connected}, we may assume that $G$ has at least two connected components.
	
	Let $G$ be the minimal counterexample to the theorem with respect to the number of connected components.
	We consider two cases.
	
	\bigskip\noindent
	{\bf Case 1.} At least one component of $G$, we call it $H$, 
		is not isomorphic to $K_4$ nor $K_{3,3}$.\quad			
		Let $n_1 = |V(G-H)|$ and $n_2 = |V(H)|$.
		We consider two subcases regarding the number of vertices of $G$.
		\begin{itemize}
			\item[$(a)$] Suppose that $n = 4t$. \quad
				
				If $n_1 = 4t_1$, for some integer $t_1$, 
				then $n_2 = 4t_2$, for some integer $t_2$, and so $t=t_1 + t_2$.				
				Suppose first that $G-H$ is not isomorphic to $2K_4$.
				Then, by the minimality of $G$,
				$G-H$ is $(t_1-1,t_1-1,t_1+1,t_1+1)$-decomposable
				and 
				$H$ is (using also Observation~\ref{obs:complement}) $(t_2+1,t_2+1,t_2-1,t_2-1)$-decomposable,
				which implies that 
				$G$ is $(t_1+t_2,t_1+t_2,t_1+t_2,t_1+t_2)$-decomposable, and hence $(t,t,t,t)$-decomposable;
				this realizes statement $(i)$.
				The statement $(ii)$ is realized, since $H$ is also $(t_2,t_2,t_2,t_2)$-decomposable,
				implying that $G$ is $(t_1-1+t_2,t_1-1+t_2,t_1+1+t_2,t_1+1+t_2)$-decomposable,
				hence $(t-1,t-1,t+1,t+1)$-decomposable.
				If $G-H$ is isomorphic to $2K_4$,
				then $t_1 = 2$ and we use the fact that $2K_4$ is $(2,2,2,2)$-decomposable,
				which settles both the case $(i)$ and the case $(ii)$,
				since $H$ is $(t_2,t_2,t_2,t_2)$-decomposable 
				and $(t_2-1,t_2-1,t_2+1,t_2+1)$-decomposable, respectively.
				
				Otherwise, $n_1 = 4t_1 + 2$, for some integer $t_1$, 
				and $n_2 = 4t_2 + 2$, for some integer $t_2$, and so $t=t_1 + t_2 + 1$.
				Now, we use that $G-H$ is $(t_1-1, t_1, t_1+1, t_1+2)$-decomposable.
				We realize the statement $(i)$ by using that
				$H$ is $(t_2+2,t_2+1,t_2,t_2-1)$-decomposable,
				which gives that $G$ is $(t_1 + t_2 + 1, t_1 + t_2+1, t_1+t_2+1,t_1+t_2+1)$-decomposable, 
				and hence $(t,t,t,t)$-decomposable.
				To realize the statement $(ii)$, 
				we make use of the fact that $H$ is $(t_2+1,t_2,t_2+1,t_2)$-decomposable,
				which gives that $G$ is $(t_1 + t_2, t_1 + t_2, t_1+t_2+2,t_1+t_2+2)$-decomposable, 
				and hence $(t-1,t-1,t+1,t+1)$-decomposable.
				
			\item[$(b)$] Suppose that $n = 4t+2$. \quad
				
				If $n_1 = 4t_1$, for some integer $t_1$, 
				then $n_2 = 4t_2 + 2$, for some integer $t_2$, and so $t=t_1 + t_2$.
				Again, suppose first that $G-H$ is not isomorphic to $2K_4$.
				Then, we exploit the fact that $G-H$ is $(t_1-1,t_1-1,t_1+1,t_1+1)$-decomposable.
				Since $H$ is $(t_2+2,t_2+1,t_2,t_2-1)$-decomposable,
				it follows that $G$ is $(t_1+t_2+1, t_1+t_2, t_1+t_2+1, t_1+t_2)$-decomposable,
				and thus $(t+1,t,t+1,t)$-decomposable, which, using Observation~\ref{obs:complement}, realizes the statement $(iii)$.
				The statement $(iv)$ is realized by the fact 
				that $H$ is $(t_2,t_2+1,t_2,t_2+1)$-decomposable,
				implying that $G$ is $(t_1+t_2-1, t_1+t_2, t_1+t_2+1, t_1+t_2+2)$-decomposable,
				and hence $(t-1,t,t+1,t+2)$-decomposable.
				If $G-H$ is isomorphic to $2K_4$, then $t_1 = 2$ and
				we again use the fact that $2K_4$ is $(2,2,2,2)$-decomposable,
				settling the cases $(iii)$ and $(iv)$
				due to $(t_2,t_2+1,t_2,t_2+1)$-decomposability
				and $(t_2-1,t_2,t_2+1,t_2+2)$-decomposability of $H$, respectively.
				
				If $n_1 = 4t_1 + 2$, for some integer $t_1$, 
				then $n_2 = 4t_2$, for some integer $t_2$, and again $t=t_1 + t_2$.
				We use the fact that $G-H$ is $(t_1-1,t_1,t_1+1,t_1+2)$-decomposable.
				Since $H$ is $(t_2+1,t_2+1,t_2-1,t_2-1)$-decomposable,
				we have that $G$ is $(t_1+t_2, t_1+t_2+1, t_1+t_2, t_1+t_2+1)$-decomposable,
				and hence $(t,t+1,t,t+1)$-decomposable, realizing the statement $(iii)$.
				Since $H$ is $(t_2,t_2,t_2,t_2)$-decomposable,
				we have that $G$ is $(t_1+t_2-1, t_1+t_2, t_1+t_2+1, t_1+t_2+2)$-decomposable,
				and hence $(t-1,t,t+1,t+2)$-decomposable, realizing the statement $(iv)$.
		\end{itemize}		
		
	\bigskip\noindent
	{\bf Case 2.} Every component of $G$ is isomorphic to $K_4$ or $K_{3,3}$.\quad
		Note first that by Observation~\ref{obs:complement}, 
		we have that any graph $H$ comprised of two isomorphic connected components $H'$,
		where $H'$ admits an $(a,b,c,d)$-decomposition such that $a + d = b + c$,
		admits a $(t',t',t',t')$-decomposition, where $t' = \frac{|V(H)|}{4}$; 
		we call such a decomposition {\em perfectly balanced}.
		As 
  certificate lists of admissible decompositions below,
  pairs of $K_4$ and of $K_{3,3}$ thus admit perfectly balanced decompositions.
		
		Next, we list possible decompositions of $K_4$ and $K_{3,3}$,
		which will be used in combinations for realizations of target decompositions of $G$.		
		It is easy to verify that $K_4$ is		
		$(a,b,c,d)$-decomposable,
		for every 
		\begin{align*}
			(a,b,c,d) \in &\set{(0,0,0,4), (0,0,2,2), (0,0,4,0), (0,1,2,1), (0,2,2,0), (0,3,0,1)}\,.
		\end{align*}

		For $K_{3,3}$, we have that it is $(a,b,c,d)$-decomposable,
		for every
		\begin{align*}
			(a,b,c,d) \in & \{
				(0,0,0,6),
				(0,0,2,4),
				(0,0,4,2),
				(0,0,6,0),
				(0,1,2,3),
				(0,1,4,1), \\
				&(0,2,2,2),
				(0,3,2,1),
				(0,4,0,2),
				(0,4,2,0),
				(1,0,3,2),	
				(1,1,3,1)\}\,.
		\end{align*}

		Now, let $k$ and $\ell$ be the number of connected components of $G$
		isomorphic to $K_4$ and $K_{3,3}$, respectively.
		We consider the cases regarding the parity of $k$ and $\ell$.
		\begin{itemize}
			\item[$(a)$] Suppose that $k$ and $\ell$ are both even. \quad		
				Then the statement $(i)$ is realized by the remark above 
				that every pair of isomorphic components of $G$ admits a perfectly 
				balanced decomposition.
				
				To show that the statement $(ii)$ can be realized,
				we consider two cases regarding $\ell$.
				Suppose first that $\ell > 0$.
				Then, we split $G$ in two parts, 
				the first part being $H_1 = kK_4 \cup (\ell-2)K_{3,3}$
				and hence having a perfectly balanced decomposition,
				and the second part being $H_2 = 2K_{3,3}$.
				For the two copies of $K_{3,3}$ we use
				a $(2,2,2,0)$-decomposition and a $(0,0,2,4)$-decomposition.
				This altogether gives a $(t-1,t-1,t+1,t+1)$-decomposition of $G$.
				
				Suppose now that $\ell = 0$ and thus $k > 2$.
				We again split $G$ in two parts, 
				the first part being $H_1 = (k-4)K_4$,
				which has a perfectly balanced decomposition,
				and the second part being $H_2 = 4K_{4}$,
				which has a $(3,3,5,5)$-decomposition
				(exploiting a $(2,2,0,0)$-decomposition,
				a $(1,0,3,0)$-decomposition, a $(0,1,2,1)$-decomposition,
				and a $(0,0,0,4)$-decomposition of $K_4$).
				
			\item[$(b)$] Suppose that $k$ and $\ell$ are both odd. \quad
				Again, we split $G$; namely, 
				into the graph $H_1 = (k-1)K_4 \cup (\ell -1)K_{3,3}$ and the graph $H_2 = K_4 \cup K_{3,3}$.
				The graph $H_1$ admits a perfectly balanced decomposition, 
				and for the graph $H_2$, 
				in order to realize the statement $(iii)$,
				we use a $(2,2,0,0)$-decomposition of $K_4$
				and a $(0,1,2,3)$-decomposition of $K_{3,3}$,
				obtaining a $(t,t+1,t,t+1)$-decomposition of $G$.
				To realize the statement $(iv)$, for $H_2$, we use
				a $(0,0,0,4)$-decomposition of $K_4$ and
				a $(1,2,3,0)$-decomposition of $K_{3,3}$,
				obtaining a $(t-1,t,t+1,t+2)$-decomposition of $G$.
				
			\item[$(c)$] Suppose that $k$ is even and $\ell$ is odd. \quad
				Note first that the statement $(iv)$ is trivially realized, 
				since $K_{3,3}$ admits a $(0,1,2,3)$-decomposition, 
				and we obtain a desired decomposition of $G$ by means of
				a perfectly balanced decomposition of $kK_4 \cup (\ell-1)K_{3,3}$
				and a $(0,1,2,3)$-decomposition of one $K_{3,3}$.
				
				Hence, we only need to realize the statement $(iii)$.
				We consider two subcases regarding $\ell$.
				
				Suppose first that $\ell > 1$.
				Then, $\ell \ge 3$ and we split $G$ into $H_1 = kK_4 \cup (\ell - 3)K_{3,3}$
				and $H_2 = 3K_{3,3}$.
				For $H_1$, we exploit a perfectly balanced decomposition,
				and for the three components of $H_2$, 
				we make use of the fact that $K_{3,3}$ is
				$(0,0,2,4)$-decomposable,
				$(2,3,0,1)$-decomposable, and
				$(2,2,2,0)$-decomposable,
				yielding a $(4,5,4,5)$-decomposition for $H_2$.
				
				Suppose now that $\ell = 1$.
				Then, $k \ge 2$.
				We split $G$ in $H_1 = (k-2)K_4$ and $H_2 = 2K_4 \cup K_{3,3}$.
				Again, for $H_1$, we use a perfectly balanced decomposition,
				and for $H_2$, 
				a $(0,0,0,4)$-decomposition and a $(1,2,1,0)$-decomposition of $K_4$, 
				and a $(2,2,2,0)$-decomposition of $K_{3,3}$,
				giving a $(3,4,3,4)$-decomposition for $H_2$.
			
			\item[$(d)$] Suppose that $k$ is odd and $\ell$ is even. \quad
				In this case, by the perfectly balanced decomposition of pairs of isomorphic components of $G$
				and the fact that $K_4$ admits a $(0,0,2,2)$-decomposition, the statement $(ii)$ is realized.
				
				We will show that except in the case with $k=3$ and $\ell = 0$, 
				we can always realize also the statement $(i)$.
				
				Suppose first that $\ell = 0$.
				Then $k \ge 5$. 
				We split $G$ into $H_1 = (k - 5)K_{4}$ and $H_2 = 5K_{4}$.
				For $H_1$, we use a perfectly balanced decomposition, and
				for the five components of $H_2$, we use 
				a $(0,0,0,4)$-decomposition,
				a $(0,1,2,1)$-decomposition,
				a $(1,0,3,0)$-decomposition, 
				and twice a $(2,2,0,0)$-decomposition, 
				yielding a $(5,5,5,5)$-de\-com\-po\-si\-tion of $H_2$.
				
				Suppose now that $\ell \ge 2$.
				We split $G$ in $H_1 = (k - 1)K_{4} \cup (\ell-2)K_{3,3}$ and $H_2 = K_{4} \cup 2K_{3,3}$.
				Again, for $H_1$, we use a perfectly balanced decomposition.
				For the unique $K_4$ component of $H_2$, we use		
				a $(0,0,0,4)$-decomposition, and for the two $K_{3,3}$ components,
				we use twice a $(2,2,2,0)$-decomposition, 
				hence obtaining a $(4,4,4,4)$-decomposition of $H_2$.
		\end{itemize}
		
		This completes the proof.
\end{proof}
It is easy to see that the four listed exceptions indeed 
do not admit the required decompositions.

\section{Concluding remarks}

As $2K_4$ admits a perfectly balanced decomposition, 
Theorem~\ref{thm:general} immediately implies Theorem~\ref{thm:main}.
It is straightforward to verify that out of the remaining exceptional cubic graphs: $K_4$, $3K_4$, $K_{3,3}$, the last one imposes the largest maximum deviation of $m(H,k)$ from $n/4$, i.e. $3/2$ (and $1$ for the remaining two graphs). Thus, Theorem~\ref{thm:general} implies also the following corollary confirming
Conjecture~\ref{conj:regular} for cubic graphs.
\begin{corollary}
	\label{cor:general}
	Every cubic graph $G$ on $n$ vertices contains a spanning subgraph $H$ 
	such that for every $k$, $0 \le k \le 3$, 
	$$
		\left | m(H,k) - \frac{n}{4} \right | \le \frac{1}{2}\,,
	$$
 unless $G$ is isomorphic to $K_4$, $3K_4$ or $K_{3,3}$, for which
 $\max_{0\le k \le 3} |m(H,k) - \frac{n}{4}| \le \frac{3}{2}$.
\end{corollary}

Alon and Wei~\cite{AloWei23} observed that the Handsaking Lemma immediately implies that $\max_{0\le k \le d} |m(H,k) - \frac{n}{d+1}|$ must be at least $1$ 
for any spanning subgraph $H$ of a $d$-regular graph $G$ of order $n$ with $\frac{n}{d+1}$ and $\big \lceil\frac{d}{2}\big\rceil$ being both odd integers. 
They however admitted that possibly the upper bound of $1$ can be provided for all 
$d$-regular graphs of sufficiently large order.
We specify this suspicion within the following bold conjecture, which additionally asserts that the condition above, 
mentioned by Alon and Wei, is the only reason for achieving the bound $1$ for large enough graphs.
\begin{conjecture}\label{conjecture:bold}
	For every $d$ there is a finite family $\mathcal{D}_d$ of exceptional graphs such that 
	every $d$-regular graph $G$ of order $n$ which does not belong to $\mathcal{D}_d$ 
	contains a spanning subgraph $H$ such that for every $k$, $0 \le k \le d$, 
 \begin{itemize}
\item if $\frac{n}{d+1}$, $\big\lceil\frac{d}{2}\big\rceil$ are odd integers, then $\left | m(H,k) - \frac{n}{d+1} \right | \le 1$;
\item otherwise, $\left | m(H,k) - \frac{n}{d+1} \right | \le \frac{d}{d+1}$ for even $d$, 
and $\left | m(H,k) - \frac{n}{d+1} \right | \le \frac{d-1}{d+1}$ for odd $d$.
 \end{itemize}
\end{conjecture}
The asserted stronger bound for odd $d$ again refers to the Handshaking Lemma. 
Note that our Corollary~\ref{cor:general} proves Conjecture~\ref{conjecture:bold} for cubic graphs. 
We however have no idea how to adapt our approach in the case of $d \ge 4$ in order to prove Conjecture~\ref{conjecture:bold} 
or the original Conjecture~\ref{conj:regular}. 
We reckon even the case of $d=4$ constitutes an interesting open problem.

It is trivial to observe that Conjecture~\ref{conjecture:bold} 
holds for $d=1$. 
For consistency, we thus conclude this paper by sketching an easy argument 
that it also holds in the case of $2$-regular graphs, with $\mathcal{D}_2=\{2C_3,2C_4\}$.
\begin{observation}
Every $2$-regular graph $G$ of order $n$ not isomorphic to $2C_3$ or $2C_4$ 
contains a spanning subgraph $H$ such that for every $k\in\{0,1,2\}$,
 \begin{itemize}
	\item if $\frac{n}{3}$ is an odd integer, then $\left | m(H,k) - \frac{n}{3} \right | \le 1$;
	\item otherwise, $\left | m(H,k) - \frac{n}{3} \right | \le \frac{2}{3}$.
 \end{itemize}
\end{observation}

\begin{proof}
\emph{(Sketch)} 
	It is straightforward to verify the assertion for $n\le 8$.
	In the remaining cases it is easy to note that one may effortlessly 
	find a subgraph $H'$ of $G$ with any fixed number $n_2\le n-4$ of $2$-vertices 
	and at most $4$ vertices of degree $1$ 
	(basically by including in $H'$ all consecutive edges around subsequent cycles 
	making up $G$ as long as necessary, with one possibly indispensable 
	jump to a different component at the end). 
	We then may easily supplement $H'$ with isolated edges to obtain a necessary 
	(even) number of $1$-vertices in the resulting $H$.
\end{proof}

\paragraph{Acknowledgement.} 
Results presented in this paper were obtained during a research visit 
of the first and the third author at the AGH University in Krakow,
and completed during the 6th Workshop on Graph Colourings in Pilsen.
B.~Lu\v{z}ar was supported by the 
Slovenian Research and Innovation Agency Program P1--0383, and the projects J1--3002 and J1--4008.
R. Sot\'{a}k was supported by the Slovak Research and Development Agency under the contracts No. APVV--19--0153 and APVV--23--0191, 
and the VEGA Research Grant 1/0574/21.
	
\bibliographystyle{plain}
{
	\bibliography{References}
}

\end{document}